\documentclass{daj}
\usepackage{amssymb,amsfonts, textcomp}
\usepackage{amsmath, amsthm}

\newcommand{\E}{{\mathbb E}}

\newcommand{\N}{{\mathbb N}} 
\newcommand{\Z}{{\mathbb Z}}

\newcommand{\CC}{{\mathcal C}} 
\newcommand{\CY}{{\mathcal Y}}

\newcommand{\inv}{^{-1}}
\newcommand{\ve}{\varepsilon}
\renewcommand{\d}{{\mathrm{d}}}
\newcommand{\ds}{\overline{\d}}
\newcommand{\wt}{\widetilde}

\newcommand{\bPhi}{{\pmb{\Phi}}}

\newcommand{\one}{{\pmb{1}}}

\newtheorem{lemma}{Lemma}
\newtheorem{proposition}[lemma]{Proposition}
\newtheorem{theorem}[lemma]{Theorem}
\newtheorem{corollary}[lemma]{Corollary}

\theoremstyle{definition}
\newtheorem*{definition*}{Definition}

\dajAUTHORdetails{%
  title = 
{A Short Proof of a  Conjecture of Erd\H os
Proved by Moreira, Richter and Robertson}
 author = {Bernard Host},
  plaintextauthor = {Bernard Host},
  runningtitle = {A Short Proof of a  Conjecture of Erd\H os}, 
  keywords = {Sumset, Erd\H os Conjecture, Positive density, F\o lner sequence},
}  

\dajEDITORdetails{%
   year={2019},
   number={19},
   received={3 May 2019},   
   published={3 December 2019},  
   doi={10.19086/da.11129},       
}   

\begin{document}

\begin{frontmatter}[classification=text]

\title{A Short Proof of a  Conjecture of Erd\H os\\
Proved by Moreira, Richter and Robertson} 

\author[bh]{Bernard Host}

\begin{abstract}
We give a short proof of  a sumset conjecture of Erd\H os, recently proved by Moreira, Richter and Robertson: every subset of the integers of positive density contains the sum of two infinite sets.  
The proof is  written in the framework of classical ergodic theory.
\end{abstract}
\end{frontmatter}

\section{Introduction}

In this note we give a short proof of a result of Moreira, Richter and Robertson, conjectured by Erd\H os:
\begin{theorem}[\cite{MRR}]
\label{th:main}
If $A\subset \N$ has positive density then there exist infinite subsets $B$ and $C$ of $\N$ with $A\supset B+C$.
\end{theorem}

Here and below, by the \emph{density}  of a subset $A$ of $\N$ we mean its  upper uniform density $\d^*(A)$
defined by
$$
\d^*(A):=\lim_{N\to\infty}\sup_{M\in\N} \frac{|A\cap [M,M+N)|}N.
$$

The proof extends to the case of countable abelian groups up to notational changes.
We believe that a  generalization of the method to countable amenable groups is possible, by using the  maximal isometric factor in place of the Kronecker factor.

The combinatorial part of our prof is closely related to the corresponding part of~\cite{MRR}: Theorem~\ref{th:one} is the dynamical counterpart of~\cite[Theorem 2.2]{MRR}. But the whole proof of Theorem~\ref{th:main} 
presented here is very different of~\cite{MRR}, as it is written inside the framework of classical ergodic theory.

In~\cite{MRR} the authors ask whether any set of positive density contains the sum of three infinite sets. A possible strategy would be to show that it contains the sum of a set of positive density and an infinite set,  and then to use Theorem~\ref{th:main}. However, this method does not work:
\begin{proposition}
\label{prop:counterexample}
There exists a set of positive density not containing any sum of a set of positive density and an infinite set.
\end{proposition}

Section~\ref{sec:preliminaries} contains some preliminaries. Theorem~\ref{th:main} is proved in Section~\ref{sec:proof} and Proposition~\ref{prop:counterexample} in Section~\ref{sec:counterexample}.



\section{Preliminaries}
\label{sec:preliminaries}
In this paper we use the convention that $\N=\{0,1,2,\dots\}$.
\subsection{F\o lner sequences}

We make a constant use of F\o lner sequences but, since we are dealing only with subsets of the integers, we never need general F\o lner sequences and we can assume that a  F\o lner sequence is a sequence $\bPhi=(\Phi_N)_{N\geq 1}$ of intervals of $\N$ whose lengths tend to $\infty$. If $A\subset \N$ we write
\begin{gather*}
\ds_\bPhi(A)=\limsup_{N\to\infty}\frac{|A\cap\Phi_N]}{|\Phi_N|};\\
d_\bPhi(A)=\lim_{N\to\infty}\frac{|A\cap\Phi_N]}{|\Phi_N|}\text{ if this limit exists.}
\end{gather*}
Therefore, 
$$
\d^*(A)=\sup_\bPhi \d_\bPhi(A)$$
where the supremum is taken for all F\o lner sequences $\bPhi$ such that $\d_\bPhi(A)$ exists, and the supremum is attained.

\subsection{Topological dynamical systems}
A topological dynamical system $(X,T)$ is a compact metric space endowed with a map $T\colon X\to X$, continuous and onto. We write $d_X$ for the distance on $X$.

Let $x_0\in X$, $\bPhi$ a F\o lner sequence and $\mu$ a probability measure on $X$. We say that $x_0$ is \emph{generic for $\mu$ along $\bPhi$} if 
$$
\frac 1{|\Phi_N|}\sum_{n\in\Phi_N}\delta_{T^nx_0}\to\mu\text{ weakly* as } N\to\infty
$$
where $\delta_x$ is the Dirac mass at $x$. This is equivalent to
$$
\text{for every }f\in\CC(X), \quad \frac 1{|\Phi_N|}\sum_{n\in\Phi_N}f(T^nx_0)\to\int f\,d\mu \text{ as }N\to\infty.
$$
In this case, $\mu$ is supported on the closed orbit of $x_0$ under $T$ and is invariant under $T$. For  $U\subset X$ we have
\begin{align*}
\ds_\bPhi\bigl(\{(n\colon T^nx_0\in U\}\bigr)&\geq\mu(U)\text{ if $U$ is open;}\\
\d_\bPhi\bigl(\{(n\colon T^nx_0\in U\}\bigr)&=\mu(U)\text{ if $U$ is clopen  (that is, open and closed).}
\end{align*}

We remark that for every $x_0\in X$, every F\o lner sequence admits a subsequence along which $x_0$ is generic for some measure $\mu$.

We say that $x_0$ is generic for $\mu$ without mention of a F\o lner sequence if it is generic for the usual F\o lner sequence $\bigl([0,N)\bigr)_{N\geq 1}$.

\subsection{The key combinatorial result}
The next theorem is the dynamical counterpart  of~\cite[Theorem~2.2]{MRR}.
\begin{theorem}
\label{th:one}
Let $(X,T)$ be a topological dynamical system, $x_0\in X$,  $E$ a clopen subset of $X$ and 
$$
A=\{n\in\N\colon T^nx_0\in E\}.
$$
Let $X\times X$ be endowed with the transformation $T\times T$.
Let $x_1\in X$ and let $\nu$ be a measure on $X\times X$, for which the point $(x_0,x_1)$ is  generic along some F\o lner sequence $\bPhi=(\Phi_N)_{N\in\N}$. Assume that there exists $\ve>0$ and a sequence $(m_i)$ of integers tending to $\infty$ with
\begin{equation}
\label{eq:one}
T^{m_i}x_0\to x_1\ \text{ and }\ \nu(T^{-m_i}E\times E)\geq\ve\text{ for every }i.
\end{equation}
Then there exist infinite subsets $B$, $C$ of $\N$ with $A\supset B+C$.
\end{theorem}

This theorem can be deduced from~\cite[Proposition 2.5]{MRR} but we give a  proof for completeness.
The construction gives that the set $B$ ban be taken as  included in any given infinite subset of $\{m_i\colon i\in\N\}$; the set $C$ can be taken as included in any infinite given subset of 
$\{\ell\in\N\colon T^\ell x_1\in E\}$.

\begin{proof}
We use the following result of Bergelson:
\begin{lemma}[\cite{B}]
\label{lem:Ber}
Let $(Y,\CY,\lambda)$ be a probability space, $\ve>0$, and $(B_n)_{n\geq 1}$ a sequence in $\CY$ with $\lambda(B_n)\geq\ve$ for every $n$. Then this sequence admits a subsequence $(B_{n_j})_{j\geq 1}$ such that
$$
\text{ for every }k\geq 1,\quad \lambda\Bigl(\bigcap_{j=1}^k B_{n_j}\Bigr)>0.
$$
\end{lemma}
Applying this lemma to the probability $\nu$ on $X\times X$ and to the sets $T^{-m_i}E\times E$, we obtain 
that the sequence $(m_i)_{i\geq 1}$  admits a subsequence, still written $(m_i)$, such that 
\begin{equation}
\label{eq:nu}
\text{for every }r\geq 1,\quad \nu\Bigl(\bigcap_{i=1}^r(T^{-m_i}E\times E)\Bigr)>0.
\end{equation}
We write
$$
L= \{\ell\in\N\colon T^\ell x_1\in E \}.
$$
For $n\in\N$, we have $(T\times T)^n(x_0,x_1)\in T^{-m_i}E\times E$ if and only if $n\in (A-m_i)\cap L$. By definition of $\nu$ and~\eqref{eq:nu},
\begin{equation}
\label{eq:inter}
\text{for every }r\geq 1,\quad \d_{\bPhi}\Bigl(\bigcap_{i=1}^r(A-m_i)\cap L\Bigr) >0.
\end{equation}
and in particular this set is infinite.

We build by induction strictly increasing sequences $(b_n)_{n\geq 1}$ and $(c_n)_{n\geq 1}$ with
$$
b_n\in\{m_i\colon i\geq 1\}\text{ and }c_n\in L\text{ for every }n;\ b_i+c_j\in A\text{ for all }i,j\geq 1.
$$

Let $n\geq 0$. If $n=0$ we chose an arbitrary $c_1\in L$.  Assume that $n\geq 1$ and that $b_1,\dots,b_n$ and $c_1,\dots,c_n$ are built with $b_i+c_j\in A$ for all $i,j\leq n$. By~\eqref{eq:inter} applied with $r=b_n$ we can chose $c_{n+1}>c_n$ with
$$
c_{n+1}\in\bigcap_{i=1}^n(A-b_i)\cap L
$$
and we have $b_i+c_{n+1}\in A$ for $i\leq n$. 

In both cases, for $j\leq n+1$ we have $c_j\in L$ that is $T^{c_j}x_1\in E$. Since $E$ is open and since $T^{m_i}x_0\to x_1$, for every sufficiently large $i$ we have 
$T^{m_i+c_j}x_0\in E$ for every $j\leq n+1$, that is $m_i+c_j\in A$. We chose $i$ with this property and moreover $m_i>b_n$ if $n\geq 1$,  and we set $b_{n+1}=m_i$.
We have $b_{n+1}+c_j\in A$ for $1\leq j\leq n+1$ and we are done.
\end{proof}

\section{Proof of Theorem~\ref{th:main}}
\label{sec:proof}

In this Section, $A$ is a subset of $\N$ with $\delta=\d^*(A)>0$, and we build the different objects appearing in the statement of Theorem~\ref{th:one} and check that they satisfy the required conditions. Theorem~\ref{th:main} then follows from Theorem~\ref{th:one}.

\subsection{Construction of a first  system}

Let $\{0,1\}^\N$ be endowed with the product topology and with the shift $T$. Elements of $\{0,1\}^\N$ are written as $x=(x(n))_{n\in\N}$. We consider $\one_A$ as an element of $\{0,1\}^\Z$ that we write $x_0$, and define 
\begin{itemize}
\item[(i)]
$X$  is the closed orbit of $x_0$ under $T$.
\end{itemize}
 Let $E$ be the cylinder set $E=\{x\in X\colon x(0)=1\}$. 
We have:
\begin{itemize}
\item[(ii)]
$E$ is a clopen subset of $X$ and 
$A=\{n\in\N\colon T^nx_0\in E\}$.
\end{itemize}

Let $\bPhi$ be a F\o lner sequence with $\d_{\bPhi}(A)=\delta$. Replacing $\bPhi$ by a subsequence we can assume that 
\begin{itemize}
\item[(iii)]
$x_0$ is generic along $\bPhi$ for some measure $\mu$ on $X$
\end{itemize}
and since $E$ is clopen we have
\begin{itemize}
\item[(iv)]
$\mu(E)\geq \delta$.
\end{itemize}

\subsection{Construction of an ergodic measure}
Let $\mu=\int_\Omega
\mu_\omega\,dP(\omega)$ be the ergodic decomposition of $\mu$ under $T$. 
 Since $\int_\Omega
\mu_\omega(E)\,dP(\omega)=\delta$, we have  $P(\{\omega\in\Omega\colon\mu'_\omega(E)\geq\delta\})>0$.
Therefore there exists an ergodic measure $\mu'$ with $\mu'(E)\geq\delta$.

By the pointwise ergodic theorem, the measure $\mu'$ admits a generic point $x_1\in X$ in the usual sense, meaning  that for every $f\in\CC(X)$ we have
$$
\frac 1N\sum_{n=0}^{N-1}f(T^nx_1)\to \int f\,d\mu\text{ as }N\to\infty.
$$
Since the orbit of $x_0$ is dense and $T$ is continuous there exists a sequence $(m_N)_{N\geq 1}$ of integers such that 
$$
\text{for every }N\text{ and }0\leq n< N,\ d_X(T^{m_N+n}x_0,T^nx_1)<1/N.
$$
For every $f\in\CC(X)$, by uniform continuity of $f$ we have
$$
\frac 1N\sum_{n=0}^{N-1}f(T^{m_N+n}x_0)\to\int f\,d\mu\text{ as }N\to\infty.
$$
This means that $x_0$ is  generic of $\mu'$ along the F\o lner sequence $\bPhi'=(\Phi'_N)_{N\geq 1}$ where $\Phi'_N=[m_N,m_N+N)$.

Substituting $\bPhi'$ for $\bPhi$ and $\mu'$ for $\mu$, the properties (i)-(iv) remain valid and moreover
\begin{itemize}
\item[(v)]
$\mu$ is ergodic under $T$.
\end{itemize}

\subsection{Using the Kronecker factor}
We recall that the Kronecker factor $(Z,m_Z,S)$ of an ergodic system $(X,\mu,T)$ is a compact abelian group $Z$ endowed with its Haar measure $m_Z$, with a translation $S$  and with a factor map $\pi\colon X\to Z$, characterized by
\begin{multline}
\label{eq:def-Kro}
\text{for all }\phi,\psi\in L^2(\mu)
\text{ with }\E_\mu(\psi\mid Z)=0
\\
\text{for every }\ve>0,\quad
\d^*\Bigl(\Bigl\{n\in\N\colon \Bigl|
\int T^n\psi\cdot \phi\,d\mu\Bigr|\geq\ve\Bigr\}\Bigr)=0.
\end{multline}
The transformation $S$ has the form
\begin{equation}
\label{eq:def-S}
\text{for every }z\in Z,\quad Sz=z+\alpha
\end{equation}
where $\alpha$ is a fixed element of $Z$.

It should be noticed that, even when $\mu$ is an invariant ergodic measure on a topological dynamical system $(X,T)$, the factor map $\pi\colon X\to Z$ is not continuous in general. The next Proposition explains how to modify the given system and obtain continuity. It is stated in~\cite{HK,HK2} for a distal system, and thus applies in particular to the isometric system $(Z,S)$.

\begin{proposition}[\mbox{\cite[Proposition 6.1]{HK}, \cite[Proposition 24.3]{HK2}}]
\label{prop:kro}
Let $(X,T)$ be a topological dynamical system, $x_0\in X$, and $\mu$ an ergodic invariant probability measure supported on the closed orbit of $x_0$ under $T$. Let $(Z,m_Z,S)$ be the Kronecker factor of $(X,\mu,T)$, with factor map $\pi\colon X\to Z$.
 Let $X\times Z$ be endowed with the transformation $T\times S$.
 Let $\wt\mu$ be the measure on $X\times Z$, image of $\mu$ under the map $x\mapsto (x,\pi(x))\colon X\to X\times Z$. Then there exists a F\o lner sequence $\wt\bPhi$ and a point $z_0$ of $Z$ such that $(x_0,z_0)$ is generic for $\wt\mu$ along $\wt\bPhi$.
\end{proposition}
 The conclusion means that, for every $f\in\CC(X)$ and every $h\in\CC(Z)$ we have
$$
\int_X f\cdot h\circ\pi\,d\mu=\lim_{N\to\infty}\frac 1{|\wt\Phi_N|}\sum_{n\in\wt\Phi_N}f(T^nx_0)\,h(S^nz_0).
$$
The first projection  $p\colon(X\times Z,T\times T)\to(X,T)$ is a factor map in the topological sense. Moreover, $p\colon(X\times Z,\wt\mu,T\times T)\to(X,\mu,T)$ is an isomorphism in the ergodic sense. In particular, $(X\times Z,\wt\mu,T\times T)$ is ergodic, its  Kronecker factor  is  $(Z,m_Z,T)$,  the factor map is the second projection   and thus is  continuous. 
The key point in the Proposition is the genericity of $(x_0,z_0)$.

\medskip

We apply this Proposition to the system $(X,\mu,T)$ and the point $x_0$ introduced above and let $(Z,m_Z,S)$, $\pi$, $\wt\mu$ and $z_0$  and $\wt\bPhi$ be given by this Proposition. Let  $\wt{x_0}=(x_0,z_0)$ and let $\wt X$ be the closed orbit of $\wt{x_0}$ in $X\times Z$.  Since the point $\wt{x_0}$ is generic for $\wt\mu$ along $\wt\bPhi$, $\wt\mu$ is supported on $\wt X$.
 Let $\wt p\colon\wt X\to X$ and $\wt\pi\colon \wt X\to Z$ be the restriction to $\wt X$ of the two natural projections $X\times Z\to X$ and $X\times Z\to Z$ respectively, and let
 $\wt E=\wt p\inv(E)$.
Then $\wt E$ is a clopen subset of $\wt X$,  $A=\{n\colon \wt T^n\wt{x_0}\in\wt E\}$ and $\d_{\wt\bPhi}(A)=\wt\mu(\wt E)=\mu(E)\geq\delta$.

 Substituting $\wt{x_0}$ for $x_0$, $\wt X$ for $X$, $\wt\bPhi$ for $\bPhi$,  $\wt\mu$ for $\mu$ and $\wt\pi$ for $\pi$, the properties (i)-(v) above remain valid and moreover,
\begin{itemize}
\item[(vi)]
The factor map $\pi\colon X\to Z$ from $X$ to the Kronecker factor $(Z,m_Z,S)$ of $(X,\mu,T)$ is continuous.
\end{itemize}
We define
$$
z_0=\pi(x_0).
$$

\subsection{Choosing a point $x_1$}

We copy the next definition and results from~\cite{L}.
\begin{definition*}[\cite{L}]
A F\o lner sequence $\bPhi=(\Phi_N)$ is \emph{tempered} if there exists a constant $C>0$ such that
$$
\Bigl|\bigcup_{k<N}(\Phi_{N}-\Phi_k)\Bigr|\leq C|\Phi_{N}|
$$
for every $N$. (Here $\Phi_N-\Phi_k=\{a-b\colon a\in\Phi_N,\ b\in\Phi_k\}$.)
\end{definition*}
\begin{lemma}[\cite{L}]
\label{lem:L}
Every F\o lner sequence admits a tempered subsequence.
\end{lemma}
\begin{theorem}[\cite{L}]
\label{th:L}
If $(X,\mu,T)$ is a system and $\bPhi$ is a tempered F\o lner sequence in $\N$ then for every $f\in L^1(\mu)$ the limit 
$$
\lim_{N\to\infty}\frac 1{|\Phi_N|}\sum_{n\in\Phi_N}f(T^nx)
$$
exists of $\mu$-almost every $x$. If $\mu$ is ergodic, then the limit is the constant $\int f\,d\mu$.
\end{theorem}

\begin{corollary}
\label{cor:L}
Let $(X,T)$ be a topological dynamical system, $\mu$ an ergodic measure on $X$ and $\bPhi$ a tempered F\o lner sequence. Then $\mu$-almost every $x\in X$ is generic for $\mu$ along $\bPhi$.
\end{corollary}
\begin{proof}
It suffices to apply Theorem~\ref{th:L} to all functions in a countable subset of $\CC(X)$, dense in $\CC(X)$ under the uniform norm.
\end{proof}

Returning to our construction, we use Lemma~\ref{lem:L} and replace the F\o lner sequence  with a tempered subsequence, still written $\bPhi$. 
We write
$$
X_1=\{x\in X\colon x\text{ is generic for $\mu$ along }\bPhi\}.
$$
 By Corollary~\ref{cor:L} we have $\mu(X_1)=1$.
We define also
$$
X_2\text{ is the topological support of }\mu.
$$
 We have $\mu(X_2)=1$. 
 Since the image of $\mu$ under $\pi$ is equal to $m_Z$ we have $\pi(X_2)=Z$.

We write
$$
\phi=\E_\mu(\one_E\mid Z)\text{ and } \psi=\one_E-\phi\circ\pi.
$$

We chose an open neighborhood $V$ of $0$ in $Z$ such that
\begin{equation}
\label{eq:continuous}
\int_Z|\phi(z)-\phi(z+\xi)|\,dm_Z(z)<\delta^2/4\text{ for every }\xi\in V
\end{equation}
and define
$$
X_3=\pi\inv(V+z_0).
$$
Note that for every $n$ we have $T^nx_0\in X_3$ if and only if $n\alpha\in V$.

$X_3$ is an open neighborhood of $x_0$ and, 
since $m_Z(V+z_0)>0$, we have $\mu(X_3)>0$ and we
can chose 
$$
x_1\in X_1\cap X_2\cap X_3.
$$
We write 
$$
z_1=\pi(x_1)\text{ and } \beta=z_1-z_0.
$$
Since  $x_1\in X_3$ we have $z_1\in V+z_0$ and 
$$
\beta\in V.
$$

\subsection{The joining $\nu$}
Replacing the F\o lner sequence $\bPhi$ by a subsequence we can assume that the point $(x_0,x_1)$ is generic along $\bPhi$ for a measure $\nu$ on $X\times X$. This measure is invariant under $T\times T$; since $x_0$ is generic for $\mu$, the image of $\nu$ under the first projection
is equal to $\mu$; since $x_1\in X_1$, $x_1$ is generic for $\mu$ too, and the second projection of $\nu$ is equal to $\mu$. We have that
$(X\times X,\nu,T\times T)$ is a joining of $(X,\mu,T)$ with itself.

Let $\lambda$ be the image of $\nu$ under the projection $\pi\times\pi\colon X\times X\to Z\times Z$. Then $(Z\times Z,\lambda,S\times S)$ is a joining of $(Z,m_Z,S)$ with itself.
Since $(x_0,x_1)$ is generic for $\nu$ along $\bPhi$, $(z_0,z_1)$ is generic for $\lambda$ along the same F\o lner sequence. Therefore, $\lambda$ is supported on a single closed orbit under $S\times S$ and, since $S\times S$ is  a translation on the compact abelian group $Z\times Z$, the measure $\lambda$ is ergodic under $S\times S$, and for $\lambda$-almost every $(z,z')$ we have $z'-z=z_1-z_0=\beta$. It follows that
for all $h,h'\in\CC(Z)$ we have 
\begin{equation}
\label{eq:lambda}
\int_{Z\times Z'} h(z)h'(z')\,d\lambda(z,z')=\int_Z h(z)h'(z+\beta)\,dm_Z(z).
\end{equation}
By density the same relation holds for $h,h'\in L^2(m_Z)$.

Let 
 $\theta\in L^2(\mu)$ be the image of $\one_E$ under the Markov operator associated to the joining $\nu$. This function is  characterized by
$$
\text{for every }h\in L^2(\mu),\ \int_{X\times X}h\otimes\one_E\,d\nu=
\int_X h\cdot\theta\,d\mu.
$$
For every $n$ we have
\begin{multline*}
\nu(T^{-n}E\times E)-\int_{X\times X}(S^n\phi\circ \pi)\otimes \one_E\,d\nu
=\int_{X\times X} T^n\psi\otimes\one_E\,d\nu\\
=\int_{X}T^n\psi\cdot \,\theta\,d\mu.\end{multline*}
Since $\E_\mu(\psi\mid Z)=0$, by the characterization~\eqref{eq:def-Kro} of the Kronecker factor
we have
\begin{equation}
\label{eq:density1}
\d^*\Bigl(\Bigl\{n\in\N\colon \Bigl|
\nu(T^{-n}E\times E)-\int_{X\times X}(S^n\phi\circ \pi)\otimes \one_E\,d\nu\Bigr|\geq\delta^2/8\Bigr\}\Bigr)=0.
\end{equation}
For every $n$, by definition of the measure $\lambda$  we have
\begin{multline*}
\int_{X\times X}(S^n\phi\circ\pi)\otimes \one_E \,d\nu
-\int_{X\times X}(S^n\phi)\otimes \phi\,d\lambda\\
=\int_{X\times X}(S^n\phi\circ\pi)\otimes \psi\,d\nu
=\int_{X\times X}(\phi\circ\pi)\otimes (T^{-n}\psi)\,d\nu.
\end{multline*}
Exchanging the role of the coordinates and proceeding as above we obtain
\begin{equation}
\label{eq:density2}
\d^*\Bigl(\Bigl\{n\in\N\colon
\Bigl|
\int_{X\times X}(S^n\phi\circ\pi)\otimes \one_E\,d\nu-\int_{Z\times Z}(S^n\phi)\otimes\phi\,
d\lambda\Bigr|\geq\delta^2/8\Bigr\}\Bigr)=0.
\end{equation}
Combining~\eqref{eq:density1} and~\eqref{eq:density2}  and using Formula~\eqref{eq:lambda} for $\lambda$ and formula~\eqref{eq:def-S} for $S$  we obtain that
$$
\d^*\Bigl(\Bigl\{n\in\N\colon \Bigl|\nu(T^{-n}E\times E)-\int_{Z}\phi(z+n\alpha)\phi(z+\beta) \,dm_Z(z)\Bigr|\geq\delta^2/4\Bigr\}\Bigr)=0.
$$
We use now the fact that $\beta\in V$, the definition of this set and the bound $\int\phi^2\,dm_Z\geq\delta^2$ and we obtain that the set
$$
\Lambda=\big\{n\in\N\colon n\alpha\in V\text{ and }\nu(T^{-n}E\times E)\leq\delta/4\bigr\}
$$
satisfies
$$
\d^*(\Lambda)=0.
$$

\subsection{Conclusion}
We check now that the point $x_1$ and the measure $\nu$ satisfy all the hypothesis needed to apply 
Theorem~\ref{th:one}. It remains only to build a good sequence $(m_i)$.

For every integer $i>0$ the set 
$$
X_3\cap \{ x\in X\colon d_X(x,x_1)<1/i\}
$$
 is an open neighborhood of $x_1$, and since $x_1$ belongs to the topological support $X_2$ of $\mu$, the measure of this set is $>0$ and thus the set
 $$
F_i=\{n\in\N\colon T^nx_0\in X_3\text{ and }d_X(T^nx_0,x_1)<1/i\}
$$
satisfies
$\ds_\bPhi(F_i)>0$. Since $\ds_\bPhi(\Lambda)=0$, $\ds_\bPhi(F_i\setminus\Lambda)>0$ and in particular $F_i\setminus\Lambda$ is infinite. We thus can chose a sequence of integers $(m_i)$, tending to $\infty$, with
$m_i\in F_i\setminus\Lambda$ for every $i$. We have that 
$$T^{m_i}x_0\to x_1\text{ as }i\to\infty$$
and, for every $i$, 
\begin{gather*}
T^{m_i}x_0\in X_3\text{ and thus }m_i\alpha\in V;\\
m_i\notin\Lambda\text{ and thus } \nu(T^{-m_i}E\times E)>\delta^2/4.
\end{gather*}

All the hypothesis of Theorem~\ref{th:one} are satisfied with $\ve=\delta^2/4$ and we are done.
\qed

\section{Proof of Proposition~\ref{prop:counterexample}}
\label{sec:counterexample}

Let $(X,T)$ be an uniquely ergodic topological dynamical system with invariant measure $\mu$ and assume that the measure preserving system $(X,\mu,T)$ is mixing. Let $U\subset X$ be an open subset with $0<\mu(U)\leq\mu(\overline U)<1$, $x_0\in X$ and 
$$
A=\{n\in\N\colon T^nx_0\in U\}.
$$

Since $\mu(U)>0$ and by unique ergodicity, $\d^*(A)>0$. We show that it does not contain a sum $B+C$ with $\d^*(B)>0$  and $C$ infinite.

Assume by contradiction that $B$ and $C$ exist withe these properties and let 
$$
K=\overline{\{T^nx_0\colon n\in B\}}.
$$
By unique ergodicity,
$$
\mu(K)\geq\d^*(\{n\in\N\colon T^nx_0\in K\})\geq \d^*(B)>0.
$$
On the other hand, for every $c\in C$ we have $T^c\{T^bx_0\colon b\in B\}\subset U$ and thus 
$$
\text{for every }c\in C,\quad T^cK\subset\overline U.
$$

Since $(X,\mu,T)$ is mixing and since $K$ and $X\setminus\overline U)$ have positive measure, for every sufficiently large $n$ we have 
$\mu(T^nK\cap(X\setminus\overline U))>0$ and thus $T^nK\not\subset \overline U$. We have a contradiction. \qed

\bibliographystyle{amsplain}

\begin{thebibliography}{99}

\bibitem{B}
 Vitaly  Bergelson.
\newblock 
Sets of recurrence of $\Z^m$-actions and properties of sets of differences in $\Z^m$.
\newblock 
{\em J. London Math. Soc. (2)}, 31(2):295–-304, 1985.

\bibitem{HK}
Bernard Host and Bryna Kra.
\newblock
Uniformity seminorms on $\ell^\infty$ and applications.
\newblock
 {\em J. Anal. Math.},  108:219--276, 2009. 
 
\bibitem{HK2}
Bernard Host and  Bryna Kra.
\newblock
{\em Nilpotent Structures in Ergodic Theory.}
\newblock
Mathematical Surveys and Monographs  236,  
American Mathematical Society, Providence, RI, 2018.

\bibitem{L}
Elon Lindenstrauss.
\newblock 
Pointwise convergence for amenable groups.
\newblock
{\em Invent. Math.},  146(2):259--295, 2001.

\bibitem[MRR]{MRR}
 Joel Moreira, Florian Karl Richter and  Donald Robertson
\newblock
A proof of a sumset conjecture of Erd\H os.
\newblock
{\em Ann. of Math. (2)}, 89(2):605-–652, 2019.

\end{thebibliography}


\begin{dajauthors}
\begin{authorinfo}[bh]
  Bernard Host\\
  Professor\\
  LAMA, Universit\'e Paris-Est Marne-la-Vall\'ee, France.
  \\
  bernard.host\imageat{}u-pem\imagedot{}fr \\
\end{authorinfo}

\end{dajauthors}

\end{document}